\documentclass[12pt]{amsart}
\usepackage{amssymb,amsmath,amsfonts,latexsym}
\usepackage{bm}
\usepackage[all,cmtip]{xy}
\usepackage{amscd}

\newtheorem{Theorem}{\sc Theorem}[section]
\newtheorem{Lemma}[Theorem]{\sc Lemma}
\newtheorem{Proposition}[Theorem]{\sc Proposition}
\newtheorem{Corollary}[Theorem]{\sc Corollary}
\theoremstyle{definition}

\newtheorem{Remark}[Theorem]{\sc Remark}

\theoremstyle{plain}

\newtheorem{thm}{Theorem}[section]

\theoremstyle{definition}

\numberwithin{equation}{section}

\begin{document}

\title[On improvements of the Hardy, Copson and Rellich inequalities]
{On improvements of the Hardy, Copson and Rellich inequalities}

\author[Das]{Bikram Das}
\address{Dr. APJ Abdul Kalam Technical University, Lucknow-226021 \&
Indian Institute of Carpet Technology, Chauri Road, Bhadohi- 221401, Uttar Pradesh, India}
\email{dasb23113@gmail.com}
\author[Manna]{Atanu Manna$^*$ }
\address{Indian Institute of Carpet Technology, Bhadohi-221401, Uttar Pradesh, India}
\email{atanu.manna@iict.ac.in, atanuiitkgp86@gmail.com ($^*$Corresponding author)}


\begin{abstract}
Using a method of factorization and by introducing a generalized discrete Dirichlet's Laplacian matrix $(-\Delta_{\Lambda})$, we establish an extended improved discrete Hardy's inequality and Rellich inequality in one dimension. We prove that the discrete Copson inequality (E.T. Copson, \emph{Notes on a series of positive terms}, J. London Math. Soc., 2 (1927), 9-12.) in one-dimension admits an improvement. We also prove that the improved Copson's weights are optimal (in fact \emph{critical}). It is shown that improvement of the Knopp inequalities (Knopp in J. London Math. Soc. 3(1928), 205-211 and 5(1930), 13-21) lies on improvement of the Rellich inequalities. Further, an improvement of the generalized Hardy's inequality (Hardy in Messanger of Math. 54(1925), 150-156) in a special case is obtained.
\end{abstract}
\maketitle

\section{Introduction}{\label{secint}}
The celebrated classical discrete Hardy's inequality (\cite{GHYLITTLE}, Theorem 326) in one dimension asserts that for a sequence $\{a_{n}\}$ of complex numbers the following
\begin{align}{\label{DHI}}
&\displaystyle\sum_{n=1}^{\infty}\Big|\frac{1}{n}\sum_{k=1}^{n}a_{k}\Big|^{2}<4\displaystyle\sum_{n=1}^{\infty}|a_{n}|^{2},
\end{align} holds unless $a_n$ is null. An equivalent version of inequality (\ref{DHI}) is reads as below:
\begin{align}{\label{EFDHI}}
&\displaystyle\sum_{n=1}^{\infty}|A_n-A_{n-1}|^{2}\geq\displaystyle\sum_{n=1}^{\infty}\frac{|A_n|^{2}}{4n^{2}},
\end{align}where $A=\{A_n\}\in C_{c}(\mathbb{N}_0)$, space of finitely supported functions on $\mathbb{N}_0=\{0, 1, 2, 3, \ldots\}$ with the assumption that $A_0=0$. It is known that the constants $`4$' in (\ref{DHI}) or $`\frac{1}{4}$' in (\ref{EFDHI}) is best possible. The Hardy inequalities were beautifully elaborated during the period $1906-1928$ of the twentieth century. There are many contributors such as E. Landau, G. P\'{o}lya, I. Schur, M. Riesz for the development of inequality (\ref{DHI}). A systematic survey article on prehistory of the Hardy's inequality is well-written by Kufner, Maligranda and Persson \cite{AKR}. As an application point of view the inequality (\ref{DHI}) has been extended in various ways. The inequality (\ref{DHI}), and its continuous analogue have major applications in spectral theory, graph theory, and differential equations. \\
In 1925, G. H. Hardy (\cite{GHYNOTE251}, \cite{GHYNOTE252}) established many interesting results. One of them is the following extension of Hardy's inequality (\ref{DHI}). Suppose that $\{q_n\}$ sequences of real numbers such that $q_{n}>0$, and denote $A_{n}=q_{1}a_{1}+ q_{2}a_{2}+ \ldots +q_{n}a_{n}$ and $Q_{n}=q_1+q_2+\ldots+q_{n}$ for $n\in \mathbb{N}$. If $\{\sqrt{q_n}a_{n}\}\in \ell_2$ then
\begin{align}{\label{GDHI}}
\displaystyle\sum_{n=1}^{\infty}q_nQ_{n}^{-2}|A_{n}|^{2}&<4\displaystyle\sum_{n=1}^{\infty}q_{n}|a_{n}|^{2},
\end{align} unless all $a_n$ is null. Also the constant term $`4$' is sharp.\\
E. T. Copson \cite{ECN} further introduced and studied an extended version of inequality (\ref{GDHI}) as below. Let $1<c\leq 2$. Then
\begin{align}{\label{COPI}}
&\displaystyle\sum_{n=1}^{\infty}q_{n}Q^{-c}_{n}|A_{n}|^{2}\leq\Big(\frac{2}{c-1}\Big)^{2}\displaystyle\sum_{n=1}^{\infty}q_{n}Q^{2-c}_{n}|a_{n}|^{2},
\end{align}where the associated constant term is best possible, equality holds good when all $a_n$ are `$0$'.\\
It is observed that the constant $`\frac{1}{4}$' in (\ref{EFDHI}) is best possible but the whole weight $`\frac{1}{4n^2}$' can not, and this surprising discovery is due to Keller, Pinchover and Pogorzelski \cite{MKR} (see also \cite{MKRGRAPH}), and they proved that the following inequality
\begin{align}{\label{IMHI2}}
\displaystyle\sum_{n=1}^{\infty}| A_n-A_{n-1}|^{2}&\geq\displaystyle\sum_{n=1}^{\infty} w_n|A_n|^{2}>\frac{1}{4}\displaystyle
\sum_{n=1}^{\infty} \frac{|A_n|^{2}}{n^{2}},
\end{align}where the weight sequence $w_n$ for which the inequality (\ref{EFDHI}) is improved defined as follows:
\begin{center}
$w_n=\frac{(-\Delta)n^{\frac{1}{2}}}{n^{\frac{1}{2}}}=2-\sqrt{1-\frac{1}{n}}-\sqrt{1+\frac{1}{n}}>\frac{1}{4n^2}$, $n\in \mathbb{N}$,
\end{center}
and the discrete Dirichlet Laplacian operator $(-\Delta)$ acting on the sequence $\{A_n\}$ was introduced (see \cite{BGDKFS}) as below:
\begin{align*}
((-\Delta) A)_{n} &= \left\{
\begin{array}{lll}
    2A_{0}-A_{1} & \quad \mbox{if~~} n=0,\\
    2A_{n}-A_{n-1}-A_{n+1} & \quad \mbox{if~~} n\in\mathbb{N}.
\end{array}\right.
\end{align*}
The authors (\cite{MKR}, \cite{MKRGRAPH}) have also shown that the weight sequence $w_n$ in inequality (\ref{IMHI2}) is optimal \emph{(critical)} which means if $\widetilde{w}\geq w$ holds point-wise for any other Hardy weight $\widetilde{w}$ then $\widetilde{w}=w$. For a more generalized notion of optimality, we refer to the articles \cite{BGDKFS}, \cite{MKRGRAPH}, and \cite{DKALFS} for the readers.\\
An elementary proof of the inequality (\ref{IMHI2}) was given by Krej\v{c}i\v{r}\'{i}k and \v{S}tampach \cite{DKK}, and a different proof of (\ref{IMHI2}) was presented by Huang in \cite{HUANG21}. By using a sequence $\mu=\{\mu_n\}$ of strictly positive real numbers such that $\mu_0=0$, an extension of the inequality (\ref{IMHI2}) was established by Krej\v{c}i\v{r}\'{i}k, Laptev and \v{S}tampach (see Theorem 10, \cite{DKALFS}) as provided below
\begin{align}{\label{IMHIGN3}}
\displaystyle\sum_{n=1}^{\infty}| A_n-A_{n-1}|^{2}&\geq \displaystyle\sum_{n=1}^{\infty} w_n(\mu)|A_n|^{2},
\end{align}
where $w_n(\mu)$ is given by
\begin{center}
$w_n(\mu)=\frac{((-\Delta)\mu)_n}{\mu_n} =2-\frac{\mu_{n-1}}{\mu_n}-\frac{\mu_{n+1}}{\mu_n}$, $n\in \mathbb{N}$.
\end{center}The authors in \cite{DKALFS} also obtained a criteria for optimality (\emph{criticality}) of $w_n(\mu)$, and some further results on it. If one chooses $\alpha=2-c$, and $q_n=1$ for all $n\in \mathbb{N}$ in (\ref{COPI}) then we obtain a power type Hardy inequality
\begin{align}{\label{COPIPAR}}
\displaystyle\sum_{n=1}^{\infty}|A_{n}-A_{n-1}|^{2}{n}^{\alpha}&\geq \frac{(\alpha-1)^2}{4}\sum_{n=1}^{\infty}\frac{|A_{n}|^{2}}{n^2}{n}^{\alpha}.
\end{align}In 2022, Gupta \cite{SGA} studied the improvement of (\ref{COPIPAR}) and proved that when $\alpha\in \{0\}\cup[1/3, 1)$, the following inequality holds:
\begin{align}{\label{GUPI}}
\displaystyle\sum_{n=1}^{\infty}|A_{n}-A_{n-1}|^{2}n^{\alpha} &\geq \displaystyle\sum_{n=1}^{\infty}w_{n}(\alpha,\beta)|A_{n}|^{2}>\frac{(\alpha-1)^2}{4}\sum_{n=1}^{\infty}\frac{|A_{n}|^{2}}{n^2}{n}^{\alpha},
\end{align}where $\alpha, \beta\in \mathbb{R}$, $w_{1}(\alpha,\beta):=1+2^{\alpha}-2^{\alpha+\beta}$, and for $ n\geq2$
\begin{align*}
w_{n}(\alpha,\beta)& =n^{\alpha}\Big[1+\Big(1+\frac{1}{n}\Big)^{\alpha}-\Big(1-\frac{1}{n}\Big)^{\beta}-\Big(1+\frac{1}{n}\Big)^{\alpha+\beta}\Big], ~~~\beta=\frac{1-\alpha}{2}.
\end{align*}
By considering all the latest developments, recently the authors in \cite{DASMAN} have obtained a generalized improved discrete Hardy's inequality in the following form
\begin{align}{\label{GIMHI}}
\displaystyle\sum_{n=1}^{\infty}\frac{|A_{n}-A_{n-1}|^{2}}{\lambda_{n}}\geq \displaystyle\sum_{n=1}^{\infty}w_{n}(\lambda, \mu )|A_n|^{2},
\end{align} where $\lambda=\{\lambda_n\}$ is a real sequence such that $\lambda_{n}>0$, $n\in \mathbb{N}$. The criteria for optimality of the weight sequence $w_{n}(\lambda, \mu)$ is also discussed in \cite{DASMAN}, where $w_{n}(\lambda, \mu)$ was defined as below
\begin{align*}
w_{n}(\lambda, \mu)& =\frac{1}{\lambda_{n}}+\frac{1}{\lambda_{n+1}}-\frac{\mu_{n-1}}{\lambda_{n}\mu_{n}}-\frac{\mu_{n+1}}{\lambda_{n+1}\mu_{n}}.
\end{align*}
Now if we choose $q_n=n$ for $n\in \mathbb{N}$ in (\ref{GDHI}), then we get a particular type of generalized Hardy inequality with sharp constant
\begin{align}{\label{GDHIQNn}}
\displaystyle\sum_{n=1}^{\infty}nQ_{n}^{-2}|A_{n}|^{2}&<4\displaystyle\sum_{n=1}^{\infty}n|a_{n}|^{2},
\end{align}where $Q_n=\frac{n(n+1)}{2}$. As far as our knowledge is concerned, we couldn't find any point-wise improvement study of (\ref{GDHIQNn}) in the literature. Therefore, first we start with the following investigation.
\begin{center}
\emph{Q(a) Is the generalized Hardy's inequality (\ref{GDHIQNn}) admits an improvement?}
\end{center}
The authors in \cite{DASMAN} also studied an improvement of the Copson inequality (\ref{COPI}) in case when $q_n=n$, and $\alpha=2-c$, and obtained an improved Copson inequality when $c=\frac{3}{2}$. Now if we choose $q_n=n^2$, and $c=3/2$ in (\ref{COPI}) then we don't have any information about improvement of the corresponding Copson inequality stated as below
\begin{align}{\label{GDHIQN2CN3BY2}}
\displaystyle\sum_{n=1}^{\infty}\frac{1}{16}\frac{n^{2}}{\widetilde{S_{n}}^{\frac{3}{2}}}|A_n|^{2}&\leq\displaystyle\sum_{n=1}^{\infty}
\sqrt{\widetilde{S}_n}\frac{|A_{n}-A_{n-1}|^{2}}{n^{2}},
\end{align}where $\widetilde{S}_n=\frac{n(n+1)(2n+1)}{6}$, and $A_{n}=a_{1}+4a_2+\ldots+n^2a_n$, $A_0=0$. Similarly, when we choose $q_n=n^3$ with $c=\frac{3}{2}$ in inequality (\ref{COPI}), then we get
\begin{align}{\label{GDHIQN3CN3BY2}}
\displaystyle\sum_{n=1}^{\infty}\frac{1}{16}\frac{n^{3}}{\widehat{S}^{\frac{3}{2}}_{n}}|A_n|^{2}
&\leq\displaystyle\sum_{n=1}^{\infty}\sqrt{\widehat{S}_n}\frac{|A_{n}-A_{n-1}|^{2}}{n^{3}},
\end{align}where $\widehat{S}_n=\Big(\frac{n(n+1)}{2}\Big)^{2}$, and $A_{n}=a_{1}+8a_2+\ldots+n^3a_n$, $A_{0}=0$.

Then it is natural to ask the following question:
\begin{center}
\emph{Q(b) Is it possible to improve the Copson inequalities (\ref{GDHIQN2CN3BY2}) and (\ref{GDHIQN3CN3BY2})?}
\end{center}
On the other hand, if we replace $(-\Delta)$ by bi-Laplacian operator $(-\Delta)^2$ then one gets the discrete Rellich inequality \cite{BGDKFS} in one dimension as below:
\begin{align}{\label{DRI}}
\displaystyle\sum_{n=2}^{\infty}((-\Delta)^2A)_n\bar{A}_n=\displaystyle\sum_{n=1}^{\infty}|((-\Delta)A)_n|^{2}\geq\frac{9}{16}\displaystyle\sum_{n=2}^{\infty}\frac{|A_{n}|^{2}}{n^4},
\end{align}where $A_0=0$, $A_1=0$, and the operator $(-\Delta)^2$ acting on the sequence $\{A_n\}$ is defined as
\begin{align*}
((-\Delta)^2 A)_{n} &= \left\{
\begin{array}{lll}
    5A_{0}-4A_{1}+A_2 & \quad \mbox{if~~} n=0,\\
    -4A_{0}+6A_{1}-4A_2+A_3 & \quad \mbox{if~~} n=1,\\
    6A_{n}-4A_{n-1}-4A_{n+1}+A_{n-2}+A_{n+2} & \quad \mbox{if~~} n\geq 2.
\end{array}\right.
\end{align*}
The discrete Rellich inequality (\ref{DRI}) with a sharp constant $\frac{9}{16}$ (as compare to the continuous analogue (\ref{CRI})), and its improvement is studied for the first time by Gerhat, Krej\v{c}i\v{r}\'{i}k, and \v{S}tampach in \cite{BGDKFS}. Although, Gupta \cite{SGAINTEGER} (see also \cite{SGALAT} for higher dimensions) obtained a similar discrete Rellich inequality (\ref{DRI}) but with a weaker constant $\frac{8}{16}$ instead of $\frac{9}{16}$.
The continuous analogue of (\ref{DRI}) was appeared in \cite{RELLICH}, and is reads as follows
\begin{align}{\label{CRI}}
\displaystyle\int_{0}^{\infty}|f''(x)|^2dx\geq\frac{9}{16}\displaystyle\int_{0}^{\infty}\frac{|f(x)|^{2}}{x^4}dx,
\end{align}where $f\in L^2(0, \infty)$ such that $f(0)=f'(0)=0$, and the associated constant term $\frac{9}{16}$ is best possible. Similar to the case of point-wise improvement of (\ref{DHI}), Gerhat, Krej\v{c}i\v{r}\'{i}k and \v{S}tampach \cite{BGDKFS} shown that there exist a weight sequence $\rho_n^{(2)}$ for which the whole weight $\frac{9}{16n^4}$ in (\ref{DRI}) can be improved. In fact, the authors in \cite{BGDKFS} proved that
\begin{align}{\label{IMDRI}}
\displaystyle\sum_{n=1}^{\infty}|((-\Delta)A)_n|^{2}& \geq \displaystyle\sum_{n=2}^{\infty}\rho_n^{(2)}\frac{|A_{n}|^{2}}{n^4}>\frac{9}{16}\displaystyle\sum_{n=2}^{\infty}\frac{|A_{n}|^{2}}{n^4},
\end{align}
where improved Rellich weights $\rho_n^{(2)}$ for $n\geq 2$ is defined as follows:
\begin{align*}
\rho_n^{(2)}=\frac{(-\Delta)^2n^{\frac{3}{2}}}{n^{\frac{3}{2}}}& = 6-4\Big(1+\frac{1}{n}\Big)^{3/2}-4\Big(1-\frac{1}{n}\Big)^{3/2}+\Big(1+\frac{2}{n}\Big)^{3/2}+\Big(1-\frac{2}{n}\Big)^{3/2}>\frac{9}{16n^4}.
\end{align*}It is shown that the Rellich weight $\rho_n^{(2)}$ is not critical but is non-attainable and optimal near infinity \cite{BGDKFS}. In a more recent preprint, the authors in \cite{BGDKFSCRITI} investigated the criticality and sub-criticality of the positive powers of the discrete Laplacian $(-\Delta)$ on the half-line. The Rellich inequality is studied in various domains such as in graph setting (see \cite{HUANGYE}, \cite{MKRRELGRAPH}), on integers and lattices (see \cite{SGALAT}, \cite{SGAINTEGER}). To establish the inequality (\ref{IMDRI}), the authors considered a method of factorization from \cite{GESZLITT} (see also \cite{GESZLITTMIWEL}) and a remainder approach (see \cite{DKK}) of proving the inequality (\ref{IMHI2}). At the end, a conjecture was given for the higher order Rellich inequalities (see Section 4, \cite{BGDKFS}). Indeed, it is conjectured that
\begin{align}{\label{IMPROR}}
\displaystyle\sum_{n=1}^{\infty}|((-\Delta)^{\alpha}A)_{n}|^2=\displaystyle\sum_{n=\alpha}^{\infty}((-\Delta)^{\alpha}A)_{n}\bar{A}_{n}&\geq\displaystyle
\sum_{n=\alpha}^{\infty}\rho^{(\alpha)}_{n}|A_{n}|^{2}
>\displaystyle\sum_{n=1}^{\infty}\frac{((2\alpha)!)^{2}}{16^{\alpha}(\alpha!)^{2}}\frac{1}{n^{2\alpha}}|A_{n}|^{2},
\end{align}where $A_{n}=0$ for $0\leq n\leq\alpha-1$ and higher order Rellich weight is given as below:
\begin{center}
$\rho^{(\alpha)}_{n}=\frac{((-\Delta)^\alpha\mu^{(\alpha)})_n}{\mu_n^{(\alpha)}}$, $\mu_n^{(\alpha)}=n^{\alpha-\frac{1}{2}}$.
\end{center}
With the above discussions, one may be interested to know answer of the following query:
\begin{center}
\emph{Q(c) Can we extend the Rellich inequality (\ref{IMDRI}) for an aribitrary sequence $\{\delta_n\}$ (say)?}
\end{center}
When passing through a literature of Hardy-type inequalities, we have found the existence of Knopp inequality \cite{KNOPP}(see also \cite{BENN}) which states that for $\alpha\geq1$ the following sharp inequality
\begin{align}{\label{KNOP1}}
&\displaystyle\sum_{n=1}^{\infty}\Big(\frac{1}{\binom{n-1+\alpha}{n-1}}\displaystyle\sum_{k=1}^{n}\binom{n-k+\alpha-1}{n-k}|a_{k}|\Big)^{2}
<\Big(\frac{\Gamma(\alpha+1)\Gamma(\frac{1}{2})}{\Gamma(\alpha+\frac{1}{2})}\Big)^2\displaystyle\sum_{n=1}^{\infty}|a_{n}|^{2},
\end{align}holds unless $a_n$ is null. Note that in case of $\alpha=1$, Knopp inequality (\ref{KNOP1}) is nothing but a classical discrete Hardy's inequality (\ref{DHI}). We now have another question, which will also be investigated here in the sequel.
\begin{center}
\emph{Q(d) Can the Knoop inequality (\ref{KNOP1}) be improved?}
\end{center}
Therefore, the main aim of this paper is to provide answer of all queries (Q(a) to Q(d)) raised above. With this aim, first we define an elongated discrete Dirichlet Laplacian $(-\Delta_{\Lambda})$ acting on a sequence $\{A_n\}$, and establish the corresponding improved discrete Hardy's inequality by using a factorization technique as considered in \cite{BGDKFS} (see also \cite{GESZLITT}, \cite{GESZLITTMIWEL}). As an application of this result, we give an answer to the Question (a). In a response to Question (b), we prove that there exists weight sequences for which improvement of the Copson inequalities (\ref{GDHIQN2CN3BY2}) and (\ref{GDHIQN3CN3BY2}) is possible. We also prove that the weight sequences in these improved Copson inequalities are optimal (\emph{critical}). Later, we define the square of Dirichlet Laplacian $(-\Delta_{\Lambda})$ through which generalized improved Rellich inequality will be established, and come up with an answer to Question (c). Finally, we give an affirmative answer to Question (d), that is we prove that the Knopp inequality (\ref{KNOP1}) admits an improvement. In particular, it is proved that improvement of the Knopp inequality lies on the improvement of the Rellich inequalities.\\
\indent The paper is arranged as follows. In Section 2, we establish an extended improved discrete Hardy's inequality in one dimension, and improve a generalized Hardy's inequality in a particular case. Section 3, deals with improvement of the Copson inequality with optimality. In Section 4, we establish an extended improved discrete Rellich inequality in one dimension for an arbitrary sequence $\{\lambda_n\}$ by using a method of factorization. Also validity of the remainder term is established by proving the existence of such $\{\lambda_n\}$. Finally, Section 5 demonstrates that Knopp inequality can be improved via Rellich inequality.

\section{Further extension of improved Hardy's inequality}
Before proving main results in this section, we first define an elongated version of the Dirichlet Laplacian $(-\Delta)$. For this, suppose that $\{\lambda_{n}\}$ is a strictly positive sequence of real numbers, and denote $\Lambda_{n}=\lambda_{1}+\ldots+\lambda_{n}$, $n\in\mathbb{N}$. Suppose further that $1<c\leq2$ and $A=\{A_{n}\}\in C_c(\mathbb{N}_0)$. The discrete operator Dirichlet Laplacian $(-\Delta_\Lambda)$ acting on $A$ is defined as below:
\begin{align*}
((-\Delta_\Lambda) A)_{n} &= \left\{
\begin{array}{lll}
    (\frac{\Lambda^{2-c}_{0}}{\lambda_{0}}+\frac{\Lambda^{2-c}_{1}}{\lambda_{1}})A_{0}-\frac{\Lambda^{2-c}_{1}}{\lambda_{1}}A_{1} & \quad \mbox{if~~} n=0,\\
    (\frac{\Lambda^{2-c}_{n+1}}{\lambda_{n+1}}+\frac{\Lambda^{2-c}_{n}}{\lambda_{n}})A_{n}-\frac{\Lambda^{2-c}_{n}}{\lambda_{n}}A_{n-1}-\frac{\Lambda^{2-c}_{n+1}}{\lambda_{n+1}}A_{n+1} & \quad \mbox{if~~} n\in\mathbb{N},
\end{array}\right.
\end{align*}where we assume that $\Lambda_0=\lambda_0$ is a strictly positive real number.\\
It is easy to observed that the corresponding infinite matrix representation of $(-\Delta_{\Lambda})$ acting on $A=\{A_n\}\in C_c(\mathbb{N})$ has the following form:
\begin{align*}
 (-\Delta_{\Lambda})=
\begin{pmatrix}
\Big(\frac{\Lambda^{2-c}_{1}}{\lambda_{1}}+\frac{\Lambda^{2-c}_{2}}{\lambda_{2}}\Big) & -\frac{\Lambda^{2-c}_{2}}{\lambda_{2}} &0 & 0 & \cdots\\
-\frac{\Lambda^{2-c}_{2}}{\lambda_{2}}  &\Big(\frac{\Lambda^{2-c}_{2}}{\lambda_{2}}+\frac{\Lambda^{2-c}_{3}}{\lambda_{3}}\Big)& -\frac{\Lambda^{2-c}_{3}}{\lambda_{3}}&0&\cdots  \\
0 & -\frac{\Lambda^{2-c}_{3}}{\lambda_{3}} & \Big(\frac{\Lambda^{2-c}_{3}}{\lambda_{3}}+\frac{\Lambda^{2-c}_{4}}{\lambda_{4}}\Big)& -\frac{\Lambda^{2-c}_{4}}{\lambda_{4}}&\cdots\\
\vdots & \vdots & \vdots &\vdots & \ddots
\end{pmatrix}
\end{align*}
In the following, we establish an extended improved discrete Hardy's inequality for a large class of sequences by using a factorization technique as considered in \cite{BGDKFS}. In what follows, we denote $E=diag\{\eta_{1},\eta_{2},\ldots\}$. Then we have the following result stated as below.
\begin{thm}{\label{THMIDHI}}
Let $A=\{A_n\}$ be any sequence of complex numbers such that $A_n\in C_c(\mathbb{N}_0)$ with $A_{0}=0$ and $\mu=\{\mu_{n}\}$ be any strictly positive sequence of real numbers. Then we have the following identity:
\begin{align}{\label{HFI}}
&A(-\Delta_\Lambda){\bar{A}}^{T}=AE\bar A^{T}+A({\overline{R}^{(1)}_{\Lambda}}^TR^{(1)}_{\Lambda})\bar A^{T},
\end{align}
where the sequence $\eta=\{\eta_{n}\}$ is defined as below:
\begin{align*}
\eta_{n}& =\frac{((-\Delta_\Lambda)\mu)_n}{\mu_n}=\frac{\Lambda^{2-c}_{n}}{\lambda_{n}}+\frac{\Lambda^{2-c}_{n+1}}{\lambda_{n+1}}-\frac{\mu_{n-1}\Lambda^{2-c}_{n}}
{\lambda_{n}\mu_{n}}-\frac{\mu_{n+1}\Lambda^{2-c}_{n+1}}{\lambda_{n+1}\mu_{n}}
\end{align*}
and $R^{(1)}_{\Lambda}$ has the following matrix representation:
\begin{align*}
 R^{(1)}_{\Lambda}=
\begin{pmatrix}
\sqrt{\frac{\mu_{2}\Lambda^{2-c}_{2}}{\mu_{1}\lambda_{2}}} & -\sqrt{\frac{\mu_{1}\Lambda^{2-c}_{2}}{\mu_{2}\lambda_{2}}} &0 & 0 & 0 & \cdots\\
0 &\sqrt{\frac{\mu_{3}\Lambda^{2-c}_{3}}{\mu_{2}\lambda_{3}}}  &-\sqrt{\frac{\mu_{2}\Lambda^{2-c}_{3}}{\mu_{3}\lambda_{3}}}&0& 0&\cdots\\
0 & 0  &\sqrt{\frac{\mu_{4}\Lambda^{2-c}_{4}}{\mu_{3}\lambda_{4}}}& -\sqrt{\frac{\mu_{3}\Lambda^{2-c}_{4}}{\mu_{4}\lambda_{4}}}&0&\cdots \\
\vdots & \vdots & \vdots &\vdots & \ddots
\end{pmatrix}.
\end{align*}The identity (\ref{HFI}) means the following inequality holds:
\begin{align*}
\displaystyle\sum_{n=1}^{\infty}\Big|\frac{A_{n-1}-A_{n}}{\sqrt\lambda_{n}}\Big|^{2}\Lambda^{2-c}_{n}& \geq
\displaystyle\sum_{n=1}^{\infty}\eta_{n}|A_{n}|^{2}.
\end{align*}
\end{thm}

\begin{proof}
It is observed that the operator $(-\Delta_\Lambda)-E$ has a tri-diagonal matrix representation. The method factorization enable us to write $(-\Delta_\Lambda)-E$ as a decomposition ${\overline{R}^{(1)}_{\Lambda}}^TR^{(1)}_{\Lambda}$ (see \cite{BGDKFS}, \cite{GESZLITT} and \cite{GESZLITTMIWEL}). To establish the identity (\ref{HFI}), it is only required to determine the exact expression for $R^{(1)}_{\Lambda}$ as the information about $(-\Delta_\Lambda)$, and $E$ is known to us. We claim that the remainder term $R^{(1)}_{\Lambda}$ acting on the sequence $\{A_{n}\}$ is of the following form:
\begin{center}
\begin{align*}
&(R^{(1)}_{\Lambda}A)_{n}=\Big(\sqrt\frac{\mu_{n+1}}{\mu_{n}\lambda_{n+1}}A_{n}-\sqrt\frac{\mu_{n}}{\mu_{n+1}\lambda_{n+1}}A_{n+1}\Big)\sqrt{\Lambda^{2-c}_{n+1}}.
\end{align*}
\end{center}
We suppose that the identity (\ref{HFI}), that is
\begin{align}{\label{HFI2}}
&\displaystyle\sum_{n=1}^{\infty}\Big|\frac{A_{n-1}-A_{n}}{\sqrt\lambda_{n}}\Big|^{2}\Lambda^{2-c}_{n}=
\displaystyle\sum_{n=1}^{\infty}\eta_{n}|A_{n}|^{2}+ \displaystyle\sum_{n=1}^{\infty}|(R^{(1)}_{\Lambda}A)_{n}|^{2}
\end{align} holds, where we assume that
\begin{align}{\label{HFR}}
&(R^{(1)}_{\Lambda}A)_{n}=\Big(\frac{\alpha_{n}A_{n}}{\sqrt\lambda_{n+1}}-\frac{A_{n+1}}{\alpha_{n}\sqrt\lambda_{n+1}}\Big)\sqrt{\Lambda^{2-c}_{n+1}}.     
\end{align}
Plugging the assumed value of $(R^{(1)}_{\Lambda}A)_{n}$ from (\ref{HFR}) in (\ref{HFI2}), then after simplifications and comparing the coefficients, we get a recurrence relation as below:
\begin{align*}
&\alpha^{2}_{1}=\frac{\mu_{2}}{\mu_{1}},
\end{align*}
and
\begin{align*}
&\frac{\mu_{n-1}\Lambda^{2-c}_{n}}{\mu_{n}\lambda_{n}}+\frac{\mu_{n+1}\Lambda^{2-c}_{n+1}}{\mu_{n}\lambda_{n+1}}=\frac{\alpha^{2}_{n}\Lambda^{2-c}_{n+1}}
{\lambda_{n+1}}+\frac{\Lambda^{2-c}_{n}}{\alpha^{2}_{n-1}\lambda_{n}} ~~~\forall ~n\geq2
\end{align*}
Using the value of $\alpha^{2}_{1}$, we get from above the value $\alpha^{2}_{2}$, and so on. In fact, we have
\begin{center}
$\alpha_{n}=\sqrt\frac{\mu_{n+1}}{\mu_{n}}$ $\forall n\in\mathbb{N}$.
\end{center}
Inserting the value of $\alpha_{n}$ in (\ref{HFR}), we get the exact expression for $R^{(1)}_{\Lambda}$, and this proves our claim.
\end{proof}
An important application of Theorem \ref{THMIDHI} is the following improvement of the generalized discrete Hardy's inequality (\ref{GDHI}) in a particular case. In fact, when one chooses
$q_{n}=n$ in (\ref{GDHI}) then we get
\begin{align}{\label{GDHIPC}}
\displaystyle\sum_{n=1}^{\infty}\frac{|A_{n}|^{2}}{n(n+1)^{2}}\leq\displaystyle\sum_{n=1}^{\infty}n|a_{n}|^{2}.
\end{align}
Surprisingly, we observed that the inequality (\ref{GDHIPC}) admits an improvement. Indeed, we have the following result.
\begin{Corollary}
Suppose that $A\in C_c(\mathbb{N}_0)$ with $A_{0}=0$. Then the following inequality holds:
\begin{align}
\displaystyle\sum_{n=1}^{\infty}\frac{|A_{n}|^{2}}{n(n+1)^{2}}<\displaystyle\sum_{n=1}^{\infty}\eta_{n}|A_{n}|^{2}\leq\displaystyle\sum_{n=1}^{\infty} n|a_{n}|^{2},
\end{align}where $\eta_{n}=\frac{1}{n^2(n+1)}>\frac{1}{n(n+1)^{2}}$, $n\in\mathbb{N}$.
\end{Corollary}
\begin{proof}
Let us put $c=2$ in (\ref{HFI2}), we get
\begin{align}{\label{GDHIPC1}}
&\displaystyle\sum_{n=1}^{\infty}\frac{\big|A_{n-1}-A_{n}\big|^{2}}{\lambda_{n}}\geq
\displaystyle\sum_{n=1}^{\infty}\eta_{n}\big|A_{n}\big|^{2},
\end{align}where $A_n=\lambda_1a_1+\ldots+\lambda_na_n$. Now choose $\lambda_{n}=n$ in (\ref{GDHIPC1}), and $\mu_n=n$, $c=2$ in $\eta_n$ of Theorem \ref{THMIDHI}, one gets
\begin{align}{\label{GDHIPC2}}
&\displaystyle\sum_{n=1}^{\infty}n\big|a_n\big|^{2}\geq
\displaystyle\sum_{n=1}^{\infty}\eta_{n}|A_{n}|^{2}>\sum_{n=1}^{\infty}\frac{|A_{n}|^{2}}{n(n+1)^{2}}.
\end{align}Hence the proof.
\end{proof}
We have another result (without proof) from Theorem \ref{THMIDHI} in case of $c=2$. Indeed, when one chooses $c=2$ then operator $(-\Delta_{\Lambda})$ reduced to the operator $(-\Delta_{\lambda})$, and we replace $R^{(1)}_{\Lambda}$ by $R^{(1)}_{\lambda}$. The following corollary demonstrates our result (see Theorem 2.1, \cite{DASMAN}).
\begin{Corollary}{\label{CORDASMAN}}
Let $A\in C_c(\mathbb{N}_0)$ be such that $A_{0}=0$. Suppose further that $\mu=\{\mu_{n}\}$ be any strictly positive sequence of real numbers and $S=diag\{\sigma_1, \sigma_2, \ldots\}$. Then we have
\begin{align*}
&A(-\Delta_\lambda)\bar A^{T}=AS\bar A^{T}+A({\overline{R}^{(1)}_{\lambda}}^TR^{(1)}_{\lambda})\bar A^{T},
\end{align*}which means the following inequality holds:
\begin{align}{\label{GIMHI}}
\displaystyle\sum_{n=1}^{\infty}\sigma_{n}|A_n|^{2}\leq\displaystyle\sum_{n=1}^{\infty}\frac{|A_{n}-A_{n-1}|^{2}}{\lambda_{n}},
\end{align} where
\begin{align*}
\sigma_{n}=\frac{((-\Delta_{\lambda})\mu)_n}{\mu_n}& =\frac{1}{\lambda_{n}}+\frac{1}{\lambda_{n+1}}-\frac{\mu_{n-1}}{\lambda_{n}\mu_{n}}-\frac{\mu_{n+1}}{\lambda_{n+1}\mu_{n}}.
\end{align*}
\end{Corollary}

\section{Improvement of the Copson inequality}
In this section, we prove that the Copson inequality (\ref{COPI}) admits an improvement in the special case of $c=\frac{3}{2}$. Indeed, we establish an improved Copson inequality (\ref{GDHIQN2CN3BY2}). The statement of our result is given as follows.
\begin{thm}{\label{THMGDHIQN2CN3BY2}}
Let $\{A_n\}$ be a sequence of complex numbers such that $A_{0}=0$. Then we have
\begin{align}{\label{ineqimvd}}
\displaystyle\sum_{n=1}^{\infty}\frac{1}{16}\frac{n^{2}}{\widetilde{S_{n}}^{\frac{3}{2}}}|A_n|^{2}<\displaystyle\sum_{n=1}^{\infty}\widetilde{V}_{n}|A_n|^{2}
&\leq\displaystyle\sum_{n=1}^{\infty}
\sqrt{\widetilde{S}_n}\frac{|A_{n}-A_{n-1}|^{2}}{n^{2}},
\end{align}where the sequence $\widetilde{V}_{n}$ is defined as below:
\begin{align*}
\widetilde{V}_{n}&= \frac{((-\Delta_\lambda)n^{3/4})}{n^{3/4}}=\frac{\sqrt{\widetilde{S}_n}}{n^{2}}+\frac{\sqrt{\widetilde{S}_{n+1}}}{(n+1)^{2}}-\frac{\sqrt{\widetilde{S}_n}}{n^{2}}
\Big(1-\frac{1}{n}\Big)^{\frac{3}{4}}-
\frac{\sqrt{\widetilde{S}_{n+1}}}{(n+1)^{2}}\Big(1+\frac{1}{n}\Big)^{\frac{3}{4}},
\end{align*}with $\lambda=\{\lambda_n\}$, $\lambda_n=\frac{n^2}{\sqrt{\widetilde{S}_n}}$, and $\widetilde{S}_n=\frac{n(n+1)(2n+1)}{6}$. Further, the improved Copson weight $\widetilde{V}_{n}$ is optimal (`critical').
\end{thm}
To prove this theorem, we require a sequence of lemma as provided below. Before stating these Lemma, we denote the following:
\begin{align*}
F(n)&=n^{\frac{1}{2}}(2n+3)^2\big(n^{\frac{3}{4}}-(n-1)^{\frac{3}{4}}\big)+(2n+1)^{3/2}\{(n+2)(2n+3)\}^{\frac{1}{2}}\big(n^{\frac{3}{4}}-(n+1)^{\frac{3}{4}}\big),\\
G(n)&=\Big\{n^{\frac{1}{2}}(2n+3)^2\Big(\frac{3}{4n^{\frac{1}{4}}}+\frac{3}{32n^{\frac{5}{4}}}\Big)+(2n+1)^{3/2}\{(n+2)(2n+3)\}^{\frac{1}{2}}\Big(\frac{3}{32n^{\frac{5}{4}}}-\frac{3}
{4n^{\frac{1}{4}}}\Big)\Big\}.
\end{align*}
\begin{Lemma}{\label{LEMGN}}
For any $n\in \mathbb{N}$, we have
$G(n)>\frac{36}{16}n^{\frac{5}{4}}$.
\end{Lemma}
\begin{proof}
A direct computation of the following difference gives
\begin{align*}
G(n)-\frac{36}{16}n^{\frac{5}{4}} &=\frac{3}{32n^{5/4}}H(n),
\end{align*}where
\begin{align*}
H(n)&=\sqrt{n}(32n^3+76n^2+84n+9)-(2n+1)(8n-1)\{(n+2)(2n+1)(2n+3)\}^{\frac{1}{2}}\\
& =\sqrt{n}\Big\{(32n^3+76n^2+84n+9)-\frac{(2n+1)(8n-1)}{n}\{n(n+2)(2n+1)(2n+3)\}^{\frac{1}{2}}\Big\}\\
& > \sqrt{n}\Big\{(32n^3+76n^2+84n+9)-\frac{(4n+2)(8n-1)}{n}(n+1)^2\Big\}~~~\mbox{(by A.M.-G.M. inequality)}\\
& =\frac{1}{\sqrt{n}}(30n^2+n+2)>0.
\end{align*}
This proves the Lemma.
\end{proof}

\begin{Lemma}{\label{LEMFNGN}}
Let $n\in\mathbb{N}$. Then we have the following inequality:
\begin{align*}
F(n)&>G(n).
\end{align*}
\end{Lemma}
\begin{proof}
We have
\begin{align*}
&F(n)=n^{3/4}\Big\{n^{\frac{1}{2}}(2n+3)^2\big(1-(1-\frac{1}{n})^{\frac{3}{4}}\big)+(2n+1)^{3/2}\sqrt{(n+2)(2n+3)}\big(1-(1+\frac{1}{n})^{\frac{3}{4}}\big)\Big\}
\end{align*}
Expanding $F(n)$ as an infinite series, we have
\begin{align*}
F(n)&=\displaystyle\sum_{k=1}^{\infty}(-1)^{k+1}\binom{\frac{3}{4}}{k}\frac{n^{\frac{3}{4}}}{n^{k}}\Big[n^{\frac{1}{2}}(2n+3)^2+(-1)^{k}(2n+1)^{3/2}\sqrt{(n+2)(2n+3)}\Big]\\
& =n^{\frac{3}{4}}\Big[n^{\frac{1}{2}}(2n+3)^2-(2n+1)^{3/2}\sqrt{(n+2)(2n+3)}\Big]\displaystyle\sum_{k\in \mbox{odd}} \binom{\frac{3}{4}}{k}\frac{1}{n^{k}}\\
& \hspace{1cm} + n^{\frac{3}{4}}\Big[n^{\frac{1}{2}}(2n+3)^2+(2n+1)^{3/2}\sqrt{(n+2)(2n+3)}\Big]\displaystyle\sum_{k\in \mbox{even}}-\binom{\frac{3}{4}}{k}\frac{1}{n^{k}}.
\end{align*}
Note that $\binom{\frac{3}{4}}{k}<0$ for all even $k\in\mathbb{N}$, and $\binom{\frac{3}{4}}{k}>0$ for all odd $k\in\mathbb{N}$. Also since $n(2n+3)^{3}>(n+2)(2n+1)^{3}$ holds for any $n\in \mathbb{N}$, so one gets
\begin{align*}
n^{\frac{1}{2}}(2n+3)^2>(2n+1)^{3/2}\sqrt{(n+2)(2n+3)},
\end{align*}which says that all terms of the R.H.S. of $F(n)$ are positive. Hence from above, we get
\begin{align*}
F(n)&>\displaystyle\sum_{k=1}^{2}(-1)^{k+1}\binom{\frac{3}{4}}{k}\frac{n^{\frac{3}{4}}}{n^{k}}\Big[n^{\frac{1}{2}}(2n+3)^2+(-1)^{k}(2n+1)^{3/2}\sqrt{(n+2)(2n+3)}\Big]\\
&=\Big\{n^{\frac{1}{2}}(2n+3)^2\Big(\frac{3}{4n^{\frac{1}{4}}}+\frac{3}{32n^{\frac{5}{4}}}\Big)+(2n+1)^{3/2}\sqrt{(n+2)(2n+3)}\Big(\frac{3}{32n^{\frac{5}{4}}}-
\frac{3}{4n^{\frac{1}{4}}}\Big)\Big\}\\
&=G(n).
\end{align*}
This proves the Lemma.
\end{proof}

\begin{Lemma}{\label{LEMVNTILDE}}
Suppose that $n\in \mathbb{N}$. Then we have
\begin{align*}
 &\widetilde{V}_{n}>\frac{1}{16}\frac{n^{2}}{\tilde{S_{n}}^{\frac{3}{2}}}.
\end{align*}
\end{Lemma}
\begin{proof}
Let us choose $f(n)=\widetilde{V}_{n}-\frac{1}{16}\frac{n^{2}}{\tilde{S_{n}}^{\frac{3}{2}}}$. It is now enough to prove that $f(n)>0$. Observed that $f(n)$ can be written in the following simplified form.
\begin{align*}
f(n)&=\frac{\{n(n+1)(2n+1)\}^{\frac{1}{2}}}{\sqrt{6}n^{2}}+\frac{\{(n+1)(n+2)(2n+3)\}^{\frac{1}{2}}}{\sqrt6(n+1)^{2}}-\frac{\{n(n+1)(2n+1)\}^{\frac{1}{2}}}{\sqrt6n^{2}}\Big(1-\frac{1}{n}\Big)^{\frac{3}{4}}\\
&\hspace{0.5cm}-\frac{\{(n+1)(n+2)(2n+3)\}^{\frac{1}{2}}}{\sqrt6(n+1)^{2}}\Big(1+\frac{1}{n}\Big)^{\frac{3}{4}}-\frac{1}{16}\frac{{6\sqrt{6}}n^{2}}{\Big(n(n+1)(2n+1)\Big)^{\frac{3}{2}}}\\
& = \frac{n^{2}(n+1)^{2}R(n)}{16 \sqrt6 n^{\frac{11}{4}}(n+1)^{2}\big(n(n+1)(2n+1)\big)^{\frac{3}{2}}},
\end{align*}where $R(n)$ is given as follows:
\begin{align*}
R(n)&=16(2n^{2}+3n+1)^{2}\big(n^{\frac{3}{4}}-(n-1)^{\frac{3}{4}}\big)+16n^{\frac{9}{4}}(2n+1)(4n^{3}+16n^{2}+19n+6)^{\frac{1}{2}}\\
&\hspace{1cm}-16(n+1)^{\frac{1}{4}}(2n^{2}+n)^{\frac{3}{2}}(2n^{3}+9n^{2}+13n+6)^{\frac{1}{2}}-36n^{\frac{11}{4}}\\
&=16(2n^{2}+3n+1)^{2}\big(n^{\frac{3}{4}}-(n-1)^{\frac{3}{4}}\big)+16n^{\frac{3}{2}}(2n+1)^{\frac{3}{2}}\sqrt{(n+2)(2n+3)}\big(n^{\frac{3}{4}}-(n+1)^{\frac{3}{4}}\big)\\
&\hspace{1cm}-36n^{\frac{11}{4}}\\
&>16(2n^{2}+3n)^{2}\big(n^{\frac{3}{4}}-(n-1)^{\frac{3}{4}}\big)+16n^{\frac{3}{2}}(2n+1)^{\frac{3}{2}}\sqrt{(n+2)(2n+3)}\big(n^{\frac{3}{4}}-(n+1)^{\frac{3}{4}}\big)\\
&\hspace{1cm}-36n^{\frac{11}{4}}\\
&=16n^{\frac{3}{2}}\Big(F(n)-\frac{36}{16}n^{\frac{5}{4}}\Big)\\
&> 16n^{\frac{3}{2}}\Big(G(n)-\frac{36}{16}n^{\frac{5}{4}}\Big) ~~~(\mbox{by Lemma \ref{LEMFNGN}})\\
& >0 ~~~(\mbox{by Lemma \ref{LEMGN}}).
\end{align*}
Hence $f(n)>0$, and this establishes the desired result.
\end{proof}

\begin{proof} of \emph{\textbf{Theorem \ref{THMGDHIQN2CN3BY2}}}:
The proof of this result is an immediate consequence of our result stated in Corollary \ref{CORDASMAN}. In fact, when we choose $\lambda_{n}=\frac{n^{2}}{\sqrt{\widetilde{S}_n}}$ and $\mu_{n}=n^{\frac{3}{4}}$, then the corresponding weight sequence $\sigma_n$ becomes $\widetilde{V}_{n}$ as given above, and we have the following identity
\begin{align}{\label{IDNCOPQN2}}
&\displaystyle\sum_{n=1}^{\infty}\frac{|A_{n}-A_{n-1}|^{2}}{n^{2}}\sqrt{\widetilde{S}_n}=\displaystyle\sum_{n=1}^{\infty}\widetilde{V}_n|A_n|^{2}+\displaystyle\sum_{n=1}^{\infty}\Big|\big(\frac{n}{n+1}\big)^{\frac{3}{8}}A_{n+1}-\big(\frac{n+1}{n}\big)^{\frac{3}{8}}A_{n}
\Big|^{2}\frac{\sqrt{\widetilde{S}_{n+1}}}{(n+1)^{2}}.
\end{align}
Hence the R.H.S. of inequality (\ref{ineqimvd}) is proved. To establish the L.H.S. of inequality (\ref{ineqimvd}), it is sufficient to prove that for all $n\in\mathbb{N}$
\begin{align*}
 &\widetilde{V}_{n}>\frac{1}{16}\frac{n^{2}}{\tilde{S_{n}}^{\frac{3}{2}}} ~~\mbox{holds},
\end{align*}which immediately follows from Lemma \ref{LEMVNTILDE}.\\
Now we prove the \emph{criticality} of the sequence $\widetilde{V}_n$. For this let $\widetilde{\widetilde{V}}_n$ is a sequence such that $\widetilde{\widetilde{V}}_n\geq\widetilde{V}_n$
holds for all $n\in \mathbb{N}$ and inequality (\ref{ineqimvd}) holds for the sequence $\widetilde{\widetilde{V}}_n$. Then we get the following:
\begin{align}{\label{INEQWIDETILDEVN}}
\displaystyle\sum_{n=1}^{\infty}\Big|\big(\frac{n}{n+1}\big)^{\frac{3}{8}}A_{n+1}-\big(\frac{n+1}{n}\big)^{\frac{3}{8}}A_{n}
\Big|^{2}\frac{\sqrt{\widetilde{S}_{n+1}}}{(n+1)^{2}} &\geq \displaystyle\sum_{n=1}^{\infty}(\widetilde{\widetilde{V}}_n-\widetilde{V}_n)|A_{n}|^{2} \geq 0.
\end{align}
Since the above inequality vanishes for $A_n=n^{3/4}$, so we redefine the sequence $A^{N}_{n}$, $N\geq2$ as $A^{N}_{n}=\gamma^{N}_{n}n^{\frac{3}{4}}$ where $\gamma^{N}_{n}$ defined as below:
\begin{align*}
\gamma^{N}_{n} &= \left\{
\begin{array}{lll}
    1 & \quad \mbox{if~~} n<N,\\
    \frac {N\sqrt[4]{\widetilde{S}_{n}}\sqrt n- n}{N\sqrt[4]{\widetilde{S}_{n}}\sqrt n} & \quad \mbox{if~~} N\leq n\leq N^{2}\\
    0 & \quad \mbox{if~~} n>N^{2}.
\end{array}\right.
\end{align*}
Now using the above sequence, we obtain
\begin{align*}
&\displaystyle\sum_{n=1}^{\infty}\Big|\big(\frac{n}{n+1}\big)^{\frac{3}{8}}A_{n+1}^N-\big(\frac{n+1}{n}\big)^{\frac{3}{8}}A_{n}^N
\Big|^{2}\frac{\sqrt{\widetilde{S}_{n+1}}}{(n+1)^{2}}\\
&=\displaystyle\sum_{n=1}^{\infty}\big(n(n+1)\big)^{\frac{3}{4}}\Big|\gamma^{N}_{n+1}-\gamma^{N}_{n}\Big|^{2}\frac{\sqrt{\widetilde{S}_{n+1}}}{(n+1)^{2}}\\
&=\frac{1}{N^{2}}\displaystyle\sum_{n=N+1}^{N^{2}}\big(n(n+1)\big)^{\frac{3}{4}}\Big|\frac{n+1}{\sqrt[4]{\widetilde{S}_{n+1}}\sqrt{n+1}}-\frac{n}{\sqrt[4]{\widetilde{S}_{n}}\sqrt n}\Big|^{2}\frac{\sqrt{\widetilde{S}_{n+1}}}{(n+1)^{2}}\\
&\leq\frac{1}{N^{2}}\displaystyle\sum_{n=N+1}^{N^{2}}\frac{n^{\frac{3}{4}}}{(n+1)^{\frac{9}{4}}}\leq\frac{1}{N^{2}}\int_{N}^{N^{2}}\frac{\sqrt{n(n+1)}}{n^{2}}dn\leq \frac{1}{N^{2}}\int_{N}^{N^{2}}\frac{2n+1}{2n^2}dn=\frac{\log N}{N^2}+\frac{1}{2N}-\frac{1}{2N^2},
\end{align*}
which tends to $0$ as $N\rightarrow\infty$. Hence from (\ref{INEQWIDETILDEVN}), we get
\begin{center}
$\displaystyle\sum_{n=1}^{\infty}(\widetilde{\widetilde{V}}_n-\widetilde{V}_n)|A_{n}^N|^{2}=0$.
\end{center}
Therefore, $\widetilde{\widetilde{V}}_n= \widetilde{V}_n$ holds for all $n\in \mathbb{N}$, hence the sequence $\widetilde{V}_n$ is \emph{critical}. This completes proof of the theorem.
\end{proof}
Now it is contemplated to establish an improved Copson inequality (\ref{GDHIQN3CN3BY2}). The improved Copson inequality (\ref{GDHIQN3CN3BY2}) with a different weight sequence $\widehat{V}_{n}$ is reads as follows.

\begin{thm}
Suppose $\{a_n\}$ is any sequence of complex numbers. Then we have
\begin{align}{\label{COPIMPQN3}}
\displaystyle\sum_{n=1}^{\infty}\frac{1}{16}\frac{n^{3}}{\widehat{S}^{\frac{3}{2}}_{n}}|A_n|^{2}<\displaystyle\sum_{n=1}^{\infty}\widehat{{V}}_{n}|A_n|^{2}
&\leq\displaystyle\sum_{n=1}^{\infty}\sqrt{\widehat{S}_n}\frac{|A_{n}-A_{n-1}|^{2}}{n^{3}},
\end{align}where the sequence $\widehat{V}_{n}$ is defined as below:
\begin{align*}
\widehat{V}_{n}&= \frac{((-\Delta_\lambda)n)}{n}=\frac{\sqrt{\widehat{S}_n}}{n^{3}}+\frac{\sqrt{\widehat{S}_{n+1}}}{(n+1)^{3}}-\frac{\sqrt{\widehat{S}_n}}{n^{3}}
\Big(1-\frac{1}{n}\Big)-\frac{\sqrt{\widehat{S}_{n+1}}}{(n+1)^{3}}\Big(1+\frac{1}{n}\Big),
\end{align*}with $\lambda=\{\lambda_n\}$, $\lambda_n=\frac{n^3}{\sqrt{\widehat{S}_n}}$, $\widehat{S}_n=\Big(\frac{n(n+1)}{2}\Big)^{2}$, and $A_{n}=a_{1}+8a_2+\ldots+n^3a_n$, $A_{0}=0$. Moreover, the improved Copson weight $\widehat{V}_{n}$ is `critical'.
\end{thm}

\begin{proof}
We omit the proof of R.H.S part of the inequality (\ref{COPIMPQN3}) as it is immediately follows from Corollary \ref{CORDASMAN} by choosing $\lambda_{n}=\frac{n^{3}}{\sqrt{\widehat{S}_n}}$ and $\mu_{n}=n$. Therefore, it is only remained to establish the L.H.S part of the inequality (\ref{COPIMPQN3}). For this, we define $h(n)$, $n\in \mathbb{N}$ as below and a direct computation gives the following.
\begin{align*}
h(n)&={\widehat{V}_{n}}-\frac{1}{16}\frac{n^{3}}{\widehat{S}^{\frac{3}{2}}_{n}}\\
&=\frac{n+1}{2n^{2}}+\frac{n+2}{2(n+1)^{2}}-\frac{n+1}{2n^{2}}\Big(1-\frac{1}{n}\Big)-\frac{n+2}{2(n+1)^{2}}\Big(1+\frac{1}{n}\Big)-\frac{1}{2(n+1)^{3}}\\
&=\frac{(n+1)^{4}-(n^{3}+n^{2})(n+2)-n^{3}}{2(n+1)^{3}n^{3}}=\frac{4n^{2}+4n+1}{2(n+1)^{3}n^{3}}>0,
\end{align*}
which proves the desired inequality. We now establish the optimality of the weight sequence $\widehat{V}_{n}$. Note that for $\lambda_{n}=\frac{n^{3}}{\sqrt{\widehat{S}_n}}$ and $\mu_{n}=n$, we get an identity from Corollary \ref{CORDASMAN} as below
\begin{align}{\label{IDNCOPQN3}}
&\displaystyle\sum_{n=1}^{\infty}\frac{|A_{n}-A_{n-1}|^{2}}{n^{3}}\sqrt{\widehat{S}_n}=\displaystyle\sum_{n=1}^{\infty}\widehat{V}_n|A_n|^{2}
+\displaystyle\sum_{n=1}^{\infty}\Big|\sqrt{\frac{n}{n+1}}A_{n+1}-\sqrt{\frac{n+1}{n}}A_{n}\Big|^{2}\frac{\sqrt{\widehat{S}_{n+1}}}{(n+1)^{3}}.
\end{align}
To establish the `\emph{criticality}' of the sequence $\widehat{V}_n$, we choose an another weight sequence $\widehat{\widehat{V}}_n$ such that $\widehat{\widehat{V}}_n\geq\widehat{V}_n$
holds for all $n\in \mathbb{N}$, and satisfy the inequality (\ref{IDNCOPQN3}. Then we have the following:
\begin{align}{\label{INEQWIDEHATVN}}
\displaystyle\sum_{n=1}^{\infty}\Big|\sqrt{\frac{n}{n+1}}A_{n+1}-\sqrt{\frac{n+1}{n}}A_{n}
\Big|^{2}\frac{\sqrt{\widehat{S}_{n+1}}}{(n+1)^{3}} &\geq \displaystyle\sum_{n=1}^{\infty}(\widehat{\widehat{V}}_n-\widehat{V}_n)|A_{n}|^{2} \geq 0.
\end{align}
Observed that the above inequality vanishes when $A_n=n$. Therefore, to regularize the sequence $A^{N}_{n}$, $N\geq2$ we redefine $A^{N}_{n}=\psi^{N}_{n}n$ where $\psi^{N}_{n}$ is defined as below:
\begin{align*}
\psi^{N}_{n} &= \left\{
\begin{array}{lll}
    1 & \quad \mbox{if~~} n<N,\\
    1-\frac {1}{Nn} & \quad \mbox{if~~} N\leq n\leq N^{2}\\
    0 & \quad \mbox{if~~} n>N^{2}.
\end{array}\right.
\end{align*}
Now using the above sequence, we get
\begin{align*}
&\displaystyle\sum_{n=1}^{\infty}\Big|\sqrt{\frac{n}{n+1}}A_{n+1}^N-\sqrt{\frac{n+1}{n}}A_{n}^N
\Big|^{2}\frac{\sqrt{\hat{S}_{n+1}}}{(n+1)^{3}}\\
&=\displaystyle\sum_{n=1}^{\infty}n(n+1)\Big|\psi^{N}_{n+1}-\psi^{N}_{n}\Big|^{2}\frac{\sqrt{\hat{S}_{n+1}}}{(n+1)^{3}}\\
&<\frac{1}{N^{2}}\displaystyle\sum_{n=N+1}^{N^{2}}\frac{1}{n}\leq\frac{1}{N^{2}}\int_{N}^{N^{2}}\frac{1}{n}dn= \frac{\log N}{N^2}\rightarrow 0 ~~\mbox{as}~~N\rightarrow\infty.
\end{align*}
Hence from inequality (\ref{INEQWIDEHATVN}), we get $\displaystyle\sum_{n=1}^{\infty}(\widehat{\widehat{V}}_n-\widehat{V}_n)|A_{n}^N|^{2}=0$. Therefore, $\widehat{\widehat{V}}_n= \widehat{V}_n$ holds for all $n\in \mathbb{N}$, hence the sequence $\widehat{V}_n$ is optimal. This completes the proof.
\end{proof}

To avoid complications, and large computations we will choose $c=2$ in the rest of the work. We note that the operator $(-\Delta_{\Lambda})$ has now become $(-\Delta_{\lambda})$, and we consider the square of the Dirichlet Laplacian operator $(-\Delta_{\lambda})^{2}$ to get extended improved discrete Rellich inequalities in one-dimension.
\section{Extension of improved Rellich inequality}
We have already mentioned in Section 1, that similar to the case of discrete Hardy's inequality, the discrete Rellich inequality (\ref{DRI}) can be improved pointwise (see \cite{BGDKFS} for details). In this section, we define a generalized bi-Laplacian operator $(-\Delta_{\lambda})^{2}$ acting on $A\in C_c(\mathbb{N}_0)$ such that $A_0=A_1=0$, and obtain an extended improved discrete Rellich inequality in one-dimension.\\
Choosing $\delta_{n}=\frac{1}{\lambda_{n}}$ for each $n\in \mathbb{N}$, and replacing $(-\Delta_{\lambda})^{2}$ by $(-\Delta_{\delta})^{2}$, where the operator $(-\Delta_{\delta})^2$ acting on the sequence $A$ is defined as below:
\begin{align*}
&((-\Delta_{\delta})^2 A)_{n} \\&= \left\{
\begin{array}{lll}
    \Big(({\delta_{0}}+{\delta_{1}})^{2}+{\delta^{2}_{1}}\Big)A_{0}-
    \Big({\delta^{2}_{1}}+\delta_{0}{\delta_{1}}+\delta_{1}{\delta_{2}}\Big)A_{1}+{\delta_{1}\delta_{2}}A_2 & \quad \mbox{if~~} n=0,\vspace{0.5cm}\\
    -\Big({2\delta^{2}_{1}}+{\delta_{0}\delta_{1}}+{\delta_{1}\delta_{2}}\Big)A_{0}+\Big(({\delta_{1}}+{\delta_{2}})^{2}+
    {\delta^{2}_{1}}+{\delta^{2}_{2}}\Big)A_{1}\\
    \hspace{0.5cm}-\Big({2\delta^{2}_{2}}+{\delta_{1}\delta_{2}}+{\delta_{2}\delta_{3}}\Big)A_2+ {\delta_{2}\delta_{3}}A_3 & \quad \mbox{if~~} n=1,\vspace{0.5cm}\\
    \Big(({\delta_{n}}+{\delta_{n+1}})^{2}+{\delta^{2}_{n}}+{\delta^{2}_{n+1}}\Big)A_{n}-
    \Big(2{\delta^{2}_{n}}+{\delta_{n}\delta_{n-1}}+{\delta_{n}\delta_{n+1}}\Big)A_{n-1}\\
    -\Big(2{\delta^{2}_{n+1}}+{\delta_{n}\delta_{n+1}}+{\delta_{n+1}\delta_{n+2}}\Big)A_{n+1}+{\delta_{n}\delta_{n-1}}A_{n-2}+
    {\delta_{n+1}\delta_{n+2}}A_{n+2} & \quad \mbox{if~~} n\geq 2.
\end{array}\right.
\end{align*}
Using the bi-Laplacian operator $(-\Delta_{\delta})^{2}$ and a sequence $\mu=\{\mu_{n}\}$ of strictly positive  of real numbers, we define an extended Rellich weight sequence $\sigma^{(2)}_{n}$ for $n\geq 2$ as below:
\begin{center}
$\sigma_n^{(2)}=\frac{((-\Delta_{\delta})^{2}\mu)_n}{\mu_n}$,
\end{center}
which means
\begin{align*}
\sigma^{(2)}_{n}&=\frac{((-\Delta_{\delta})^{2}\mu)_n}{\mu_n}\\
&=\Big((\delta_{n}+\delta_{n+1})^{2}+\delta^{2}_{n}+\delta^{2}_{n+1}\Big)
-\frac{\mu_{n-1}}{\mu_{n}}\Big(2\delta^{2}_{n}+\delta_{n}\delta_{n-1}+\delta_{n}\delta_{n+1}\Big)\\
&-\frac{\mu_{n+1}}{\mu_{n}}\Big(2\delta^{2}_{n+1}+\delta_{n}\delta_{n+1}+\delta_{n+1}\delta_{n+2}\Big)
+\frac{\mu_{n-2}}{\mu_{n}}\delta_{n}\delta_{n-1}+\frac{\mu_{n+2}}{\mu_{n}}\delta_{n+1}\delta_{n+2}.\\
\end{align*}
We are now ready to establish the extended version of the discrete Rellich inequality in one dimension. Denote $\Sigma=diag\{\sigma_1^{(2)}, \sigma_2^{(2)}, \ldots\}$. Then we have the following theorem, statement of which is given as follows:
\begin{thm}{\label{THMI}}
Suppose that $A\in C_c(\mathbb{N}_0)$ such that $A_0=A_1=0$. Then there exists a remainder matrix $R^{(2)}_{\delta}$ for which the following identity holds:
\begin{align}{\label{RFI}}
&A(-\Delta_{\delta})^{2}\bar A^{T}=A\Sigma \bar A^{T}+ A({\overline{R}^{(2)}_{\delta}}^TR^{(2)}_{\delta})\bar A^{T},
\end{align}which means the following inequality is true:
\begin{align}{\label{RFINQ2}}
\displaystyle\sum_{n=2}^{\infty}((-\Delta_{\delta})^{2}A)_{n}\bar{A}_n=\displaystyle\sum_{n=1}^{\infty}|(-\Delta_{\delta}A)_{n}|^{2}\geq\displaystyle
\sum_{n=1}^{\infty}\sigma^{(2)}_{n}|A_{n}|^{2}
\end{align}
\end{thm}

\begin{proof}
To prove inequality (\ref{RFINQ2}), we assume that the above identity (\ref{RFI}) holds true for unknown $R^{(2)}_{\delta}$:
\begin{align}{\label{RFI2}}
\displaystyle\sum_{n=1}^{\infty}|(-\Delta_{\delta}A)_{n}|^{2}=\displaystyle\sum_{n=1}^{\infty}
\sigma^{(2)}_{n}|A_{n}|^{2}+\displaystyle\sum_{n=1}^{\infty}|(R^{(2)}_{\delta}A)_{n}|^{2},
\end{align}where it is supposed that $R^{(2)}_{\delta}$ acting on the sequence $\{A_{n}\}$ has the following form:
\begin{align}{\label{RFR}}
&(R^{(2)}_{\delta}A)_{n}=
\gamma_{n}\sqrt{\delta_{n+1}\delta_{n+2}}A_{n}-\beta_{n}\sqrt{\delta_{n+1}}A_{n+1}+\frac{\sqrt{\delta_{n+1}\delta_{n+2}}}{\gamma_{n}}A_{n+2}, ~n\in \mathbb{N},
\end{align}
where $\beta_{n}, \gamma_{n}\in\mathbb{R}$ and $\gamma_n\neq 0$ for $n\in\mathbb{N}$. Our aim is to determine the coefficients $\beta_{n}, \gamma_{n}$ such that the identity (\ref{RFI2}) is being satisfied. To reach the goal, we compute each terms of the identity (\ref{RFI2}). Throughout our investigation in this section, it is assumed that $A_0=A_1=0$. Let us start evaluation of the following infinite sum.
\begin{align*}
&\displaystyle\sum_{n=1}^{\infty}|((-\Delta_{\delta})A)_{n}|^{2}\\
&=\displaystyle\sum_{n=1}^{\infty}|(\delta_{n}+\delta_{n+1})A_{n}
-A_{n-1}\delta_{n}-A_{n+1}\delta_{n+1}|^{2}\\
&=\displaystyle\sum_{n=2}^{\infty}\Big[
(\delta_{n}+\delta_{n+1})^{2}+\delta^{2}_{n}+\delta^{2}_{n+1}\Big]|A_{n}|^{2}
-2\mathbb{R}\displaystyle\sum_{n=2}^{\infty}\Big(\delta^{2}_{n+1}+\delta_{n+1}\delta_{n+2}\Big)\bar A_{n+1}A_{n}\\
&-2\mathbb{R}\displaystyle\sum_{n=2}^{\infty}\Big(\delta_{n}\delta_{n+1}+\delta^{2}_{n+1}\Big)\bar A_{n+1}A_{n}
+2\mathbb{R}\displaystyle\sum_{n=2}^{\infty}\delta_{n+1}\delta_{n+2}A_{n}\bar A_{n+2}.
\end{align*}
Now the following computation gives
\begin{align*}
&\displaystyle\sum_{n=1}^{\infty}|(R^{(2)}_{\delta}A)_{n}|^{2}\\
&= \displaystyle\sum_{n=1}^{\infty} |\gamma_{n}\sqrt{\delta_{n+1}\delta_{n+2}}A_{n}-\beta_{n}\sqrt{\delta_{n+1}}A_{n+1}+\frac{\sqrt{\delta_{n+1}\delta_{n+2}}}{\gamma_{n}}A_{n+2}|^{2}\\
&=\displaystyle\sum_{n=1}^{\infty}\Big[\gamma^{2}_{n}\delta_{n+1}\delta_{n+2}|A_{n}|^{2}+
\beta^{2}_{n}\delta_{n+1}|A_{n+1}|^{2}+\frac{\delta_{n+1}\delta_{n+2}}{\gamma^{2}_{n}}|A_{n+2}|^{2}
-2\mathbb{R}\Big(\gamma_{n}\beta_{n}{\sqrt{\delta_{n+2}}\delta_{n+1}}A_{n}\bar A_{n+1}\Big)\\
&-2\mathbb{R}\Big(\frac{\beta_{n}}{\gamma_{n}}\delta_{n+1}\sqrt{\delta_{n+2}} A_{n+1}\bar A_{n+2}\Big)+2\mathbb{R}\Big(\delta_{n+1}\delta_{n+2}A_{n}\bar A_{n+2}\Big)\Big]\\
& = \Big(\gamma^{2}_{2}\delta_{4}\delta_{3}+\beta^{2}_{1}\delta_{2}\Big)|A_{2}|^{2}+\displaystyle\sum_{n=3}^{\infty}
\Big[\delta_{n+1}\delta_{n+2}\gamma^{2}_{n}+
\beta^{2}_{n-1}\delta_{n}+\frac{1}{\gamma^{2}_{n-2}}\delta_{n-1}\delta_{n}\Big]|A_{n}|^{2}\\
&-2\mathbb{R}\displaystyle\sum_{n=2}^{\infty}
\Big(\gamma_{n}\beta_{n}\delta_{n+1}\sqrt{\delta_{n+2}}
+\frac{\beta_{n-1}}{\gamma_{n-1}}\delta_{n}\sqrt{\delta_{n+1}}\Big)A_{n}\bar A_{n+1}
+2\mathbb{R}\displaystyle\sum_{n=2}^{\infty}
A_{n}\bar A_{n+2}\delta_{n+1}\delta_{n+2}.
\end{align*}
Plugging these values in (\ref{RFI2}), and comparing the coefficients in both sides we get a set of equations as prescribed below:
\begin{align}{\label{delta3E}}
&(\delta_{2}+\delta_{3})^{2}+\delta^{2}_{2}+\delta^{2}_{3}=\sigma^{(2)}_{2}+
\gamma^{2}_{2}\delta_{3}\delta_{4}+\beta^{2}_{1}\delta_{2} \nonumber\\
&(\delta_{n+1}+\delta_{n})^{2}+\delta^{2}_{n}+\delta^{2}_{n+1}=\sigma^{(2)}_{n}+
\gamma^{2}_{n}\delta_{n+1}\delta_{n+2}+\beta^{2}_{n-1}\delta_{n}+\frac{1}{\gamma^{2}_{n-2}}\delta_{n-1}\delta_{n}~~\forall n\geq3 \nonumber\\
&2\delta^{2}_{n+1}+\delta_{n}\delta_{n+1}+\delta_{n+1}\delta_{n+2}
=\gamma_{n}\beta_{n}\delta_{n+1}\sqrt{\delta_{n+2}}+\frac{\beta_{n-1}}{\gamma_{n-1}}\delta_{n}\sqrt{\delta_{n+1}} ~~\forall n\geq2.
\end{align}
As we have seen that in case of discrete Hardy inequality (\ref{HFR}), the remainder term $R^{(1)}_{\lambda}$ vanishes on $\{\mu_n\}$ that is $(R^{(1)}_{\lambda}\mu)_{n}=0$. Therefore, we here additionally assume that the remainder term $R^{(2)}_{\delta}$ will also vanish when acting on $\{\mu_n\}$ that is $(R^{(2)}_{\delta}\mu)_{n}=0$ in case of discrete Rellich inequality, which means that
\begin{align}{\label{RFRSE}}
&\gamma_{n}\mu_{n}\sqrt{\delta_{n+1}\delta_{n+2}}-\beta_{n}\mu_{n+1}\sqrt{\delta_{n+1}}+
\frac{\mu_{n+2}}{\gamma_{n}}\sqrt{\delta_{n+1}\delta_{n+2}}=0,~~n\in \mathbb{N}.
\end{align}
Computing  $\beta_{n}$ from (\ref{RFRSE}) and using it in (\ref{delta3E}), we get a recurrence relation of $\gamma^{2}_{n}$ as given below\\
\begin{align}{\label{RecurrenceR}}
\gamma^{2}_{n} &= \frac{1}{\delta_{n+1}\delta_{n+2}}\frac{\mu_{n+1}}{\mu_{n}}\Big[2\delta^{2}_{n+1}+
\delta_{n}\delta_{n+1}+\delta_{n+1}\delta_{n+2} \nonumber\\
&-\frac{\mu_{n+2}}{\mu_{n+1}}\delta_{n+1}\delta_{n+2} -\frac{\mu_{n-1}}{\mu_{n}}\delta_{n}\delta_{n+1}-\frac{\mu_{n+1}}{\mu_{n}\gamma^{2}_{n-1}}\delta_{n}\delta_{n+1}\Big], ~~~n\geq 2.
\end{align}
We now choose an initial assumption for determining $\gamma^{2}_{n}$, $n\geq 2$ as
\begin{align}{\label{Initialcon}}
&\gamma^{2}_{1}=\frac{\mu_{2}}{\mu_{1}}\Big[2\frac{\delta_{2}}{\delta_{3}}+\frac{\delta_{1}}{\delta_{3}}+1\Big]-\frac{\mu_{3}}{\mu_{1}}
\end{align}
Using the value of $\gamma^{2}_{1}$, and recurrence relation (\ref{RecurrenceR}), one obtains the values of $\{\gamma^{2}_{n}\}$ for $n\geq2$. Once we get the positive values of $\{\gamma^{2}_{n}\}$, then we also get the value of $\{\beta_{n}\}$ from equation (\ref{RFRSE}). Hence we obtain a valid expression for the remainder term $(R^{(2)}_{\delta}A)_{n}$ from (\ref{RFR}).
\end{proof}

It is clear that the exact expression for $\{\gamma^{2}_{n}\}$ is very difficult to obtain but they are exists from our previous investigation. Our next result shows that such $\{\gamma^{2}_{n}\}$ will indeed exist for a suitable $\delta_{n}$ and $\mu_{n}$. It is pertinent to mention here that Gerhat et al. \cite{BGDKFS} also established such kind of existence for $\{\gamma^{2}_{n}\}$, but their result is a particular case of our result for $\delta_n=1$ $(\mbox{or}~~\lambda_n=1)$, $n\in \mathbb{N}$. Here we prove that $\{\gamma^{2}_{n}\}$ exists for a special choices of $\delta_n$ and lies within the same bounds as provided in \cite{BGDKFS}. Our result reads as follows.
\begin{Lemma}{\label{EXISTLEM}}
Let $n\in\mathbb{N}$, $p_{n}=\frac{\mu_{n+1}}{\mu_{n}}$, $\mu_n=n^{3/2}$ and $\{\gamma^{2}_{n}\}$, $\{\gamma^{2}_{1}\}$ are given in equations (\ref{RecurrenceR}) and (\ref{Initialcon}), respectively with $\delta_{n}=\frac{n+2}{n+1}$. Then
\begin{align*}
&p_{n}p_{n+1}<\gamma^{2}_{n}<p_{n}p_{n+1}p_{n+2}.
\end{align*}
\end{Lemma}
To establish this lemma, we need some results, which are given in the form of several Lemmas. Before proceeding further, we denote the following:
\begin{align}{\label{EXPTN}}
T(n)&=\Big[1+\frac{2(n+3)^{2}}{(n+2)(n+4)}+\frac{(n+3)(n+2)}{(n+4)(n+1)}
    -\frac{2(n+3)(n+2)}{(n+4)(n+1)}(\frac{n-1}{n})^{\frac{3}{2}}-2(\frac{n+2}{n+1})^{\frac{3}{2}}\Big].
\end{align}

\begin{Lemma}{\label{LEMTN}}
Suppose that $n\geq 2$, $n\in\mathbb{N}$. Then $T(n)>0$, where $T(n)$ is given in the equation (\ref{EXPTN}).
\end{Lemma}
\begin{proof}
A direct simplification of $T(n)$ gives the following, i.e,
\begin{align*}
T(n)=&\frac{1}{(n+4)(n+2)(n^{2}+n)^{\frac{3}{2}}}\Big[\sqrt{n^{2}+n}\Big(4n^{4}+28n^{3}+60n^{2}+38n\Big)\\
& -2\sqrt{n^{2}-1}\Big(n^{4}+6n^{3}+9n^{2}-4n-12\Big)
-2n\sqrt{n^{2}+2n}\Big(n^{3}+8n^{2}+20n+16\Big)\Big].
\end{align*}Applying A.M.-G.M.-H.M. inequality one gets
\begin{align*}
T(n)&>\frac{1}{(n+4)(n+2)(2n+1)(n^{2}+n)^{3/2}}\Big[2n(n+1)\Big(4n^{4}+28n^{3}+60n^{2}+38n\Big)\\
&-2n(2n+1)\Big(n^{4}+6n^{3}+9n^{2}-4n-12\Big)-2n\sqrt{(n^{2}+2n)}(2n+1)\Big(n^{3}+8n^{2}+20n+16\Big)\Big]\\
=&\frac{2n}{(n+4)(n+2)(2n+1)(n^{2}+n)^{3/2}}\Big[\Big(2n^{5}+19n^{4}+64n^{3}+97n^{2}+66n+12\Big)\\
&-\sqrt{(n^{2}+2n)}\Big(2n^{4}+17n^{3}+48n^{2}+52n+16\Big)\Big]\\
=&\Big\{\frac{2n\Big(18n^{7}+230n^{6}+1164n^{5}+3001n^{4}+4196n^{3}+3100n^{2}+1072n+144\Big)}{(n+4)(n+2)(2n+1)(n^{2}+n)^{3/2}}\Big\}D_n>0,
\end{align*} where
\begin{align*}
1/D_n&=\Big(2n^{5}+19n^{4}+64n^{3}+97n^{2}+66n+12\Big)\\
&\hspace{0.5cm}+\sqrt{(n^{2}+2n)}\Big(2n^{4}+17n^{3}+48n^{2}+52n+16\Big).\\
\end{align*}
Hence $T(n)>0$, $\forall n\geq2$.
\end{proof}

To establish the next result, we denote $U(n)$ for each $n\in \mathbb{N}$ as below
\begin{align}{\label{EQNUN}}
U({n})&=\Big\{1+2\frac{(n+3)^{2}}{(n+2)(n+4)}+\frac{(n+3)(n+2)}{(n+4)(n+1)}-\frac{(n+3)(n+2)}{(n+4)(n+1)}\Big(\frac{n-1}{n}\Big)^{\frac{3}{2}}
\Big(\frac{n+1}{n+2}\Big)^{\frac{3}{2}}\nonumber\\
    &\hspace{0.5cm}-\frac{(n+3)(n+2)}{(n+4)(n+1)}\Big(\frac{n-1}{n}\Big)^{\frac{3}{2}}
    -\Big(\frac{n+2}{n+1}\Big)^{\frac{3}{2}}\Big\}.
\end{align}

\begin{Lemma}{\label{LEMUN}}
Let $n\geq 2$ be a natural number. Then $U(n)<p_{n+1}p_{n+2}$ holds for all $n\geq 2$.
\end{Lemma}
\begin{proof}
For this, we denote $S(n)$ by the following difference
\begin{align*}
S(n)&=U(n)-p_{n+1}p_{n+2}\\
&=\Big[2\frac{(n+3)^{2}}{(n+2)(n+4)}+\frac{(n+3)(n+2)}{(n+4)(n+1)}
 +1-\frac{(n+3)(n+2)}{(n+4)(n+1)}(\frac{n-1}{n})^{\frac{3}{2}}(\frac{n+1}{n+2})^{\frac{3}{2}}\\
&\hspace{0.5cm}-\frac{(n+3)(n+2)}{(n+4)(n+1)}(\frac{n-1}{n})^{\frac{3}{2}}
-(\frac{n+2}{n+1})^{\frac{3}{2}} -(\frac{n+3}{n+1})^{\frac{3}{2}}\Big]
\end{align*}
Simplifying $S(n)$ as a fraction, we obtain a numerator $f(n)$ (say) as below:
\begin{align*}
f(n)&=\Big[2(n+3)^{2}n(n+1)\sqrt{n(n+1)(n+2)}+n(n+3)(n+2)^{2}\sqrt{n(n+1)(n+2)}\\
&\hspace{0.5cm}+n(n+1)(n+2)(n+4)\sqrt{n(n+1)(n+2)}-n(n+2)^{3}(n+4)\sqrt{n}\\
&\hspace{0.5cm}-(n+3)(n+2)(n^{2}-1)(n+1)\sqrt{n-1}-n(n+2)(n+3)(n+4)\sqrt{n(n+2)(n+3)}\\
&\hspace{0.5cm}-(n+3)(n+2)^{2}(n-1)\sqrt{(n+1)(n+2)(n-1)}\Big],
\end{align*}which further becomes
\begin{align*}
f(n)&=\sqrt{n(n+1)(n+2)}(4n^{4}+28n^{3}+60n^{2}+38n)\\
&\hspace{0.5cm} -\sqrt{n(n+2)(n+3)}(n^{4}+9n^{3}+26n^{2}+24n)\\
&\hspace{0.5cm}-\sqrt{(n-1)(n+1)(n+2)}(n^{4}+6n^{3}+9n^{2}-4n-12)\\
&\hspace{0.5cm}-\sqrt{n}(n^{5}+10n^{4}+36n^{3}+56n^{2}+32n)-\sqrt{n-1}(n^{5}+6n^{4}+10n^{3}-11n-6).
\end{align*}
We need to show that $f(n)<0$, which gives $S(n)<0$, and hence we get the desired inequality.
Now by applying A.M.-G.M.-H.M. inequality $\forall n\geq2$, we have
\begin{align*}
&(i)\hspace{0.5cm}\sqrt{(n+1)(n+2)}<\frac{2n+3}{2}\\
&(ii)\hspace{0.5cm}\frac{2(n+2)(n+3)}{2n+5}<\sqrt{(n+2)(n+3)}\\
&(iii)\hspace{0.5cm}\frac{2(n+1)(n+2)}{2n+3}<\sqrt{(n+1)(n+2)}.
\end{align*}Using these inequalities, for $n\geq 2$ we get
\begin{align*}
(2n+3)(2n+5)f(n)&<(2n+3)\sqrt{n}\Big(4n^{6}+35n^{5}+98n^{4}+58n^{3}-142n^{2}-163n\Big)\\
&\hspace{0.5cm}-(2n+5)\sqrt{n-1}\Big(4n^{6}+33n^{5}+96n^{4}+100n^{3}-34n^{2}-133n-66\Big).
\end{align*}
Since the sum of the two terms of R.H.S. of the above inequality is positive, so multiplying the same in both sides of above, we get (for $n\geq 2$)
\begin{align*}
&(2n+3)(2n+5)\Big[(2n+3)\sqrt{n}\Big(4n^{6}+35n^{5}+98n^{4}+58n^{3}-142n^{2}-163n\Big)\\
&\hspace{0.5cm}+(2n+5)\sqrt{n-1}\Big(4n^{6}+33n^{5}+96n^{4}+100n^{3}-34n^{2}-133n-66\Big)\Big]f(n)\\
&<-\Big(192n^{13}+3252n^{12}+22956n^{11}+84779n^{10}+159528n^{9}+73848n^{8}-267088n^{7}\\
&\hspace{0.5cm}-428550n^{6}+75848n^{5}+621720n^{4}+273984n^{3}-396949n^{2}-417120n-108900\Big)<0.
\end{align*}
This shows that $f(n)<0$ for $n\geq 2$. Hence for $n\geq 2$ one gets $S(n)<0$, which finishes proof of the lemma.
\end{proof}

\begin{proof} \textbf{of Lemma {\ref{EXISTLEM}}:}\\
It is observed that when $\mu_{n}=n^{\frac{3}{2}}$ then from equation (\ref{RecurrenceR}), we have
\begin{align*}
\gamma^{2}_{n}&=(\frac{n+1}{n})^{\frac{3}{2}}\Big[1+2\frac{(n+3)^{2}}{(n+2)(n+4)}+\frac{(n+3)(n+2)}{(n+4)(n+1)}-\frac{(n+3)(n+2)}{(n+4)(n+1)}(\frac{n-1}{n})^{\frac{3}{2}}\\
&-(\frac{n+2}{n+1})^{\frac{3}{2}}-\frac{(n+2)(n+3)}{(n+4)(n+1)\gamma^{2}_{n-1}}(\frac{n+1}{n})^{\frac{3}{2}}\Big].
\end{align*}
Also $\gamma^{2}_{1}=2\sqrt2\big[2\frac{16}{15}+\frac{12}{10}+1\Big]-3\sqrt3\approx 7.06036>5.19616=p_{1}p_{2}$. Now suppose that $\gamma^{2}_{n-1}>p_{n-1}p_{n}$ holds true for $n\geq 2$. Then by induction from the above expression for $\gamma^{2}_{n}$, we get
\begin{align*}
\gamma^{2}_{n}&>(\frac{n+1}{n})^{\frac{3}{2}}\Big[1+2\frac{(n+3)^{2}}{(n+2)(n+4)}+\frac{(n+3)(n+2)}{(n+4)(n+1)}
-2\frac{(n+3)(n+2)}{(n+4)(n+1)}(\frac{n-1}{n})^{\frac{3}{2}}-(\frac{n+2}{n+1})^{\frac{3}{2}}\Big]\\
&= (\frac{n+1}{n})^{\frac{3}{2}}T(n),
\end{align*}
where $T(n)$ is given in (\ref{EXPTN}). It is now sufficient to show that $T(n)>0$ for all $n\geq 2$. From Lemma \ref{LEMTN}, we see that $T(n)>0$ for all $n\geq 2$. This proves that $\gamma^{2}_{n}>p_{n}p_{n+1}$ $\forall n\in\mathbb{N}$.\\
To establish the reverse inequality, we first note that $\gamma^{2}_{1}\approx 7.06036<8=p_{1}p_{2}p_{3}$. Let us suppose that $\gamma^{2}_{n-1}<p_{n-1}p_{n}p_{n+1}$ holds $\forall n\geq2$, $n\in\mathbb{N}$. Then by induction, we get \\
\begin{align*}
\gamma^{2}_{n}&<(\frac{n+1}{n})^{\frac{3}{2}}\Big[1+2\frac{(n+3)^{2}}{(n+2)(n+4)}+\frac{(n+3)(n+2)}{(n+4)(n+1)}
-\frac{(n+3)(n+2)}{(n+4)(n+1)}(\frac{n-1}{n})^{\frac{3}{2}}(\frac{n+1}{n+2})^{\frac{3}{2}}\\
&-\frac{(n+3)(n+2)}{(n+4)(n+1)}(\frac{n-1}{n})^{\frac{3}{2}}
-(\frac{n+2}{n+1})^{\frac{3}{2}}\Big]\\
&=p_{n}U(n),
\end{align*}
where $U(n)$ is given in equation (\ref{EQNUN}). It is now enough to prove that $U(n)<p_{n+1}p_{n+2}$ for $n\geq 2$, and which follows from Lemma \ref{LEMUN}. Therefore $\gamma^{2}_{n}<p_{n}p_{n+1}p_{n+2}$, $\forall n\in\mathbb{N}$. This completes the proof of the lemma.
\end{proof}
\begin{Remark}
There exists infinitely many sequences for which $\gamma^{2}_{n}$ exists and lies within the same bounds, that is $p_{n}p_{n+1}<\gamma^{2}_{n}<p_{n}p_{n+1}p_{n+2}$, $\forall n\in\mathbb{N}$.
\end{Remark}

\section{Improvement of Knopp inequality via Rellich inequality}
In this section, our goal is to establish an improvement of the Knopp inequality (\ref{KNOP1}).
\subsection{Knopp inequality of order $\alpha=2$}
Let us substitute $\alpha=2$ in (\ref{KNOP1}), then the Knopp inequality of order $2$ reads as below
\begin{align}{\label{KNOPORDER2}}
&\displaystyle\sum_{n=1}^{\infty}\Big(\frac{1}{n(n+1)}\displaystyle\sum_{k=1}^{n}(n-k+1)|a_{k}|\Big)^{2}
<\frac{16}{9}\displaystyle\sum_{n=1}^{\infty}|a_{n}|^{2}.
\end{align}
Now we define a backward difference operator $\nabla$ acting on the sequence $\{A_n\}$ as $(\nabla A)_n=A_n-A_{n-1}$ so that $(\nabla^2 A)_n=-2A_{n-1}+A_{n-2}+A_{n}$ with $A_0=0$ and all negative subscript are zero. With this notation, and by choosing $A_{n}=\displaystyle\sum_{k=1}^{n}(n-k+1)|a_{k}|$, inequality (\ref{KNOPORDER2}) takes the following form:
\begin{align}{\label{EFKNOPORDER2}}
&\displaystyle\sum_{n=1}^{\infty}\frac{|A_n|^2}{n^{2}(n+1)^{2}}
<\frac{16}{9}\displaystyle\sum_{n=1}^{\infty}|(\nabla^2A)_{n}|^{2}.
\end{align}Then we have the following result.
\begin{Proposition}
The improvement of the Rellich inequality implies that the improvement of the Knopp inequality of order $2$.
\end{Proposition}
\begin{proof}
Suppose that the improvement of the Rellich inequality holds, i.e., inequality (\ref{IMDRI}) is satisfied as below:
\begin{align*}
\displaystyle\sum_{n=1}^{\infty}|((-\Delta)A)_n|^{2}& \geq \displaystyle\sum_{n=2}^{\infty}\rho_n^{(2)}\frac{|A_{n}|^{2}}{n^4}>\frac{9}{16}\displaystyle\sum_{n=2}^{\infty}\frac{|A_{n}|^{2}}{n^4},
\end{align*}where $A_0=A_1=0$, and $\rho_n^{(2)}$ is the improved Rellich weight. On the other hand, with this assumption we have observed that
\begin{center}
$\displaystyle\sum_{n=1}^{\infty}|(\nabla^{2}A)_{n}|^{2}=\displaystyle\sum_{n=1}^{\infty}|2A_{n-1}-A_{n-2}-A_{n}|^{2}
=\displaystyle\sum_{n=1}^{\infty}|((-\Delta) A)_{n}|^{2}$.
\end{center}Hence from inequality (\ref{EFKNOPORDER2}), one gets
\begin{align*}
\displaystyle\sum_{n=1}^{\infty}|(\nabla^{2}A)_{n}|^{2}&=\displaystyle\sum_{n=1}^{\infty}|((-\Delta) A)_{n}|^{2}\geq \displaystyle\sum_{n=2}^{\infty}\rho_n^{(2)}\frac{|A_{n}|^{2}}{n^4}>\frac{9}{16}\displaystyle\sum_{n=2}^{\infty}\frac{|A_{n}|^{2}}{n^4}\geq \frac{9}{16}\displaystyle\sum_{n=1}^{\infty}\frac{|A_n|^2}{n^{2}(n+1)^{2}}.
\end{align*}This proves the result.
\end{proof}

\subsection{Knopp inequality of order $\alpha\geq 1$}
Similar to the previous discussion, we choose
\begin{center}
$A_{n}=\displaystyle\sum_{k=1}^{n}\binom{n-k+\alpha-1}{n-k}|a_{k}|$,
\end{center}and define backward difference operator of order $\alpha$ as
\begin{center}
$(\nabla^\alpha A)_n=\displaystyle\sum_{k=1}^{n}(-1)^{k+1}\binom{\alpha}{k-1}A_{n-k+1}$.
\end{center}
We then have the higher order Knoop's inequality (\ref{KNOP1}) stated as below
\begin{align}{\label{KNOPPORDALPH}}
&\displaystyle\sum_{n=1}^{\infty}\frac{|A_{n}|^{2}}{\Big\{\binom{n-1+\alpha}{n-1}\Big\}^2}
<\Big\{\frac{\Gamma(\alpha+1)\Gamma(\frac{1}{2})}{\Gamma(\alpha+\frac{1}{2})}\Big\}^{2}\displaystyle\sum_{n=1}^{\infty}|\nabla^{\alpha}A_{n}|^{2}.
\end{align}
Then we have the following result.
\begin{Proposition}
The improvement of the higher order Knopp inequalities lies on the improvement of the higher order Rellich inequalities.
\end{Proposition}
\begin{proof}
Suppose that the inequality (\ref{IMPROR}) holds true with the assumption $A_n=0$ for all $n=0, 1, \ldots, \alpha-1$, and for all negative subscripts $n$.
Then after sincere investigation for even $\alpha\geq2$ it is observed that
\begin{align}
&\displaystyle\sum_{n=\alpha}^{\infty}|\nabla^{\alpha}A_{n}|^{2}=\displaystyle\sum_{n=\frac{\alpha}{2}}^{\infty}|((-\Delta)^{\frac{\alpha}{2}}A)_{n}|^{2}
\end{align} and for odd $\alpha\geq 2$, we have
\begin{align}
&\displaystyle\sum_{n=\alpha}^{\infty}|\nabla^{\alpha}A_{n}|^{2}=\displaystyle\sum_{n=\frac{\alpha+1}{2}}^{\infty}|\nabla((-\Delta)^{\frac{\alpha-1}{2}}A)_{n}|^{2}.
\end{align}
On the other hand, a straightforward calculation suggest that for any $\alpha\in\mathbb{N}$, the following identity holds.
\begin{align}{\label{Knopp=Rellich R}}
&\displaystyle\sum_{n=\alpha}^{\infty}|\nabla^{\alpha}A_{n}|^{2}=\displaystyle\sum_{n=\alpha}^{\infty}((-\Delta)^{\alpha}A)_{n}\bar{A}_{n}.
\end{align}
Therefore, we conclude that
\begin{align*}
\displaystyle\sum_{n=\alpha}^{\infty}|\nabla^{\alpha}A_{n}|^{2}&=\displaystyle\sum_{n=\alpha}^{\infty}((-\Delta)^{\alpha}A)_{n}\bar{A}_{n}\\
&\geq\displaystyle\sum_{n=\alpha}^{\infty}\rho^{(\alpha)}_{n}|A_{n}|^{2}
>\displaystyle\sum_{n=\alpha}^{\infty}\frac{((2\alpha)!)^{2}}{16^{\alpha}(\alpha!)^{2}}\frac{1}{n^{2\alpha}}|A_{n}|^{2}
\geq \Big\{\frac{\Gamma(\alpha+\frac{1}{2})}{\Gamma(\alpha+1)\Gamma(\frac{1}{2})}\Big\}^{2}\displaystyle\sum_{n=\alpha}^{\infty}\frac{|A_{n}|^{2}}{\Big\{\binom{n-1+\alpha}{n-1}\Big\}^{2}}.
\end{align*}
Hence the result.
\end{proof}
%


\textbf{Declarations:} The authors have no competing interests to declare that are relevant to the content of this article.

\textit{Acknowledgement:} There are currently no funding resources for this research.

 \end{document}